\documentclass[runningheads]{llncs}

\usepackage[T1]{fontenc}
\usepackage{graphicx}
\usepackage{comment}
\usepackage{tikz}
\usetikzlibrary{decorations.markings,arrows.meta}
\usepackage{biblatex}
\usepackage{cancel}
\usepackage{amssymb}
\usepackage{mathrsfs}
\usepackage{bbm}
\usepackage{amsmath} 
\usepackage{array}
\usepackage[linesnumbered,ruled,vlined]{algorithm2e}
\usepackage{doi}
\newcommand{\T}{\mathcal{T}}
\newcommand{\underlineT}{\underline{T}}
\newcommand{\overlineT}{\overline{T}}

\newcommand{\N}{\mathbb{N}}
\newcommand{\R}{\mathbb{R}}

\newcommand{\one}{\mathbbm{1}}
\newcommand{\PP}{\mathbb{P}}
\newcommand{\B}{\mathscr{B}}
\newcommand{\U}{\mathscr{U}}

\newcommand{\C}{\mathcal{C}}

\newcommand{\A}{\mathscr{A}}

\newcommand{\ZZ}{\mathcal{Z}}
\newcommand{\Pcal}{\mathcal{P}}
\newcommand{\EE}{\mathbb{E}}

\newcommand{\uppereach}{ \rightharpoonup }
\newcommand{\lowereach}{ \rightharpoondown 
  }

\newcommand{\notlowereach}{ \not\rightharpoondown
  }

\newcommand{\coloneqq}{:=}

\DeclareMathOperator{\vect}{vec}

\DeclareMathOperator{\ch}{co}

\begin{document}

\title{Upper Expected Meeting Times for Interdependent Stochastic Agents}

\author{Marco Sangalli\orcidID{0009-0005-9047-5839} \and Erik Quaeghebeur\orcidID{0000-0003-1462-401X} \and
Thomas Krak\orcidID{0000-0002-6182-7285}}

\institute{Uncertainty in AI, Eindhoven University of Technology, The Netherlands}
\maketitle              
\begin{abstract}
We analyse the problem of meeting times for interdependent stochastic agents: random walkers whose behaviour is stochastic but controlled by their selections from some set of allowed actions, and the inference problem of when these agents are all in the same state for the first time. 
We consider the case where we are epistemically uncertain about the selected actions of these agents, and show how their behaviour can be modelled using imprecise Markov chains. This allows us to use results and algorithms from the literature, to exactly compute bounds on their meeting time, which are tight with respect to our epistemic uncertainty models. After focussing on the two-agent case, we analyse and discuss how it can be extended to an arbitrary number of agents, and how the corresponding combinatorial explosion can be partly mitigated by exploiting symmetries inherent in the problem.

\end{abstract}

\section{Introduction}
Markov chains are a cornerstone of stochastic modelling, finding applications in fields as diverse as statistical physics~\cite{Metropolis1953}, queuing theory~\cite{Kelly1979}, and network science~\cite{VolchenkovBlanchard2008}. A fundamental question in this context is the \emph{hitting time}: the first moment at which a stochastic process enters a given set of states~\cite{norrisbook1997markov}.
When considering multiple independent Markov chains evolving on the same finite state space, a natural extension is the \emph{meeting time}, defined as the first moment at which all chains occupy the same state. Equivalently, this can be viewed as a hitting-time problem on the Cartesian product of their state spaces, with the target set being the diagonal~\cite{bullo2018meeting, norrisbook1997markov, lindvall2002lectures}. 

In this work, we generalize this problem to that of computing the expected meeting time of what we call \emph{interdependent stochastic agents}. Broadly, these are random walkers on some shared state space, whose behaviour is stochastic, but controlled by their selection of transition probabilities from some set representing their allowed actions. Crucially, we allow the selections that these agents make to depend on the state(s) of the other agent(s) in the system; this induces the interdependency.

On top of this construction, we consider the case where we are \emph{epistemically uncertain} about the selections that the agents use, and show that meaningful and conservative estimates for the expected meeting time can still be computed in this setting. We focus in particular on three such uncertainty models: the degenerate belief, which simply describes the case where we \emph{do} know the selections exactly; the vacuous belief, in which we are fully ignorant; and the degenerate-vacuous mixture, which combines these other models and may be particularly useful in some practical settings.

We show how the vacuous belief model allows us to describe the joint behaviour of the agents as an \emph{imprecise Markov chain}~\cite{HermansSkulj2014, hartfiel_seneta_1994, DeCooman2009} on their product space. This allows us to leverage known results from the literature, and in particular iterative algorithms for computing upper- and lower expected hitting times~\cite{krak2019hitting}, to exactly solve the expected meeting time problem in this setting.

We largely focus on the two-agent case, but show how the formalization can be naturally extended to an arbitrary number of agents using a $k$-fold product construction. Moreover, we discuss how we can partly mitigate the combinatorial explosion of this problem, by quotienting the (exponentially large) product space over the symmetries obtained by permuting the order of the agents.

The remainder of this paper is structured as follows. Section~\ref{sec:preli} reviews preliminary notions of precise and imprecise Markov chains, known theory about hitting- and meeting times, and 
presents an extension of Krak’s characterizations and algorithms~\cite{krak2021comphit} to hold under weaker conditions than their original assumptions.
Section~\ref{sec:upmeetingtime} formalizes the setting of two interdependent stochastic agents with epistemic uncertainty models for their selections, and the characterization of meeting times in terms of the hitting times of an imprecise Markov chain on the product space.
In Section~\ref{sec:manyagents} we extend the theory to an arbitrary number of interdependent stochastic agents, and Section~\ref{sec:conclusion} concludes the paper and gives directions for future work.

\section{Preliminaries}\label{sec:preli}
This section introduces the basic concepts of stochastic processes and Markov chains~\cite{norrisbook1997markov}, hitting and meeting times for Markov chains~\cite{norrisbook1997markov,bullo2018meeting, Doeblin1937} and imprecise Markov chains~\cite{HermansSkulj2014, hartfiel_seneta_1994, DeCooman2009,krak2019hitting}.

\subsection{Stochastic Processes and Markov Chains}
Let us denote by $\N$ the set of positive integers and define $\N_0\coloneqq\N\cup\{0\}$.
Let $\ZZ$ be a finite state space of cardinality $N\geq2$.  A discrete‐time stochastic process on $\ZZ$ is a sequence of $\ZZ$–valued random variables \((Z_n)_{n\in \N_0}\), and we write \(\PP_Z\) for its associated probability measure.  Throughout, we use the notations “\((Z_n)\)” and “\(\PP_Z\)” interchangeably to refer to this process.

The process $(Z_n)$ is called a \emph{Markov chain} if it satisfies the Markov property, i.e. if for every $n\in \N_0$ and every $z_0,\dots,z_{n+1}\in\ZZ$, it holds that
\begin{equation*}
   \PP_Z\bigl(Z_{n+1}=z_{n+1}\mid Z_0=z_0,\dots,Z_n=z_n\bigr)
   =
   \PP_Z\bigl(Z_{n+1}=z_{n+1}\mid Z_n=z_n\bigr).
\end{equation*}

The chain $(Z_n)$ is called (time‐)\emph{homogeneous} if the one‐step transition probabilities do not depend on $n$.  Equivalently, there exists a single matrix \( T = (T(z,z'))_{z,z'\in\ZZ}, \)
with
\[
   T(z,z')=\PP_Z\bigl(Z_{n+1}=z' \mid Z_n=z\bigr)
   \quad\text{for all }n\in\N_0,
\]
and each row of $T$ summing to unity.  
Similarly, a non-homogeneous Markov chain $(Z_n)$ is identified by a family of transition matrices $(T_n)$.

\subsection{Hitting and Meeting Times in the Precise Setting}\label{sec:hitandmeetime}
\subsubsection{Hitting Times.}
Let $(Z_n)$ be a homogeneous Markov chain on $\ZZ$ with transition matrix $T=(T(z,z'))_{z,z'\in\ZZ}$.  For any target set $A\subset\ZZ$, define the \emph{hitting time}
\[
  \tau_A :=\inf\{\,n\ge0: Z_n\in A\}\,\in\N_0\cup\{+\infty\}.
\]
This random variable represents the number of steps the process needs to do to first reach $A$. 
Conditioned on the chain starting at $z\in\ZZ$, the \emph{expected hitting time} is
\[
  h_A^T(z):=\EE_{\PP_Z}\bigl[\tau_A \mid Z_0=z\bigr].
\]
The quantity $h^{T}_A(z)$\footnote{The superscript $T$ does \textit{not} denote matrix transpose; rather, it indicates the dependency of the hitting time on the transition matrix $T$.} represents the expected number of steps that the process $(Z_n)$, starting in $z\in \ZZ$, takes before reaching the set $A$.
 
\begin{theorem}[\cite{norrisbook1997markov}]
    Let $A\subset \ZZ$ and let $(Z_n)$ be a homogeneous Markov chain with transition matrix $T$. The vector of expected hitting times $h^T_A=(h_A^T(z))_{z\in\ZZ}$ is the minimal non-negative solution to the following system of equations:
    \begin{equation}
        \begin{cases}
            h^T_A(z)=0 & \text{ for all $z\in A$,}\\
            h^T_A(z)=1+\sum_{z'\in \ZZ}T(z,z')h^T_A(z') & \text{ for all $z\notin A$,} 
        \end{cases}
    \end{equation}
    which can be rewritten as 
    \begin{equation}\label{eq:hittingtime_comp}
        h^T_A=\one_{A^c}+\one_{A^c}\cdot Th^T_A,
    \end{equation}
    where $\one$ is the indicator function and $\cdot$ represents the element-wise multiplication.
\end{theorem}

\subsubsection{Meeting Times.} 
Consider a joint stochastic process $(X_n,Y_n)_{n\in \N_0}$ on \(\ZZ^2:=\ZZ \times \ZZ\).  Define the \emph{meeting time} for $X_n$ and $Y_n$ as
\[
  \mu :=\inf\{\,n\ge0: X_n = Y_n\}\,\in\N_0\cup\{+\infty\}.
\]
If the joint process starts at $(X_0,Y_0)=(x,y)$ then the \emph{expected meeting time} is
\begin{equation} \label{eq:defexpectedmeetingtime}
      m{(x,y)}\coloneqq\EE_{\PP_{(X,Y)}}\bigl[\mu \mid X_0=x,\;Y_0=y\bigr].
\end{equation}
The quantity $m{(x,y)}$ represents the expected number of steps that the two processes, when starting in $x$ and $y$ respectively, take before meeting. 
In the case where the joint process $(X_n, Y_n)$ is obtained as the independent product of two homogenous Markov chains $(X_n)$ and $(Y_n)$ with transition matrices $T$ and $S$, respectively, the expected meeting time has a particularly nice characterisation.
\begin{theorem}[{\cite{bullo2018meeting,Doeblin1937}}]
\label{thm:meeting}
Let $(X_n)$ and $(Y_n)$ be two homogenous Markov chains with transition matrices $T$ and $S$ respectively. The matrix of expected meeting times $m=(m{(x,y)})_{x,y\in\ZZ}$ is the minimal non-negative solution to the following system of equations:
    \begin{equation}\label{eq:meetingsys}
    \begin{cases}
        m{(z,z)}=0 & \text{for all $z\in \ZZ$,}\\
        m{(x,y)}=1+\sum_{x',y'\in \ZZ}T{(x,x')}S{(y, y')} m{(x',y')} & \text{for all $x\ne y$,}
    \end{cases}
    \end{equation}
which can be rewritten as 
    \begin{equation}\label{eq:meetime_matrix}
        m=J+J\cdot \left(TmS^\top\right),
    \end{equation}
    where $J$ is a $N\times N$ matrix that has zeros on the main diagonal and ones everywhere else, and $S^\top$ is the transpose of $S$.
\end{theorem}

The meeting time problem can be reformulated as a hitting time problem on the product space $\ZZ^2$, with target set being \(D := \{(z, z) \mid z \in \ZZ\}\) and transition probability from state $(x,y)$ to state $(x',y ')$ being $T{(x,x')}S{(y,y')}$.

Using the operation of vectorization and the Kronecker product~\cite{Zehfuss1858, , bullo2018meeting}, denoted by the symbol $\otimes$, we may rewrite \eqref{eq:meetime_matrix} as
\begin{equation*}
    \vect(m)=\one_{D^c}+\one_{D^c}\cdot (T\otimes S)\vect(m),
\end{equation*}
where the operator $\vect$ stacks the columns of $m$ to make an $N^2$ long vector.

\subsection{Imprecise Markov Chains} \label{subsec:IMCpreli}
Let us consider a set of transition matrices $\T$ on the finite space $\ZZ$, with $|\ZZ|=N$.
Throughout this paper, we assume that $\T$ is a non-empty, compact, and convex set with rows that are independently specified—a property known in the literature as having separately specified rows (SSR)~\cite{HermansSkulj2014, krak2019hitting}. 
The assumption of having separately specified rows is equivalent to defining $N$ separate convex and compact sets of probability distributions, one for each state $z\in \ZZ$.

An imprecise Markov chain (IMC) can be viewed as a collection of Markov chains.  A natural example is
\(
\Pcal^H_{\T}
\), 
the collection of all homogeneous Markov chains whose transition matrix belongs to \(\T\). 
We are interested in computing hitting times for imprecise Markov chains, so we introduce the following upper and lower expectation operators:
\begin{align*}
    \overline{\EE}_\T[\,\cdot\,]\coloneqq\sup_{\PP\in \Pcal^H_\T}\EE_{\PP}[\,\cdot\,] \, , &\qquad 
    \underline{\EE}_\T[\,\cdot\,]\coloneqq\inf_{\PP\in \Pcal^H_\T}\EE_{\PP}[\,\cdot\,] \,.
\end{align*}
These definitions allow us to define the upper and lower expected hitting time of a target set $A\subset \ZZ$ as 
\begin{align*}
\quad\overline{h}^{\T}_A(z) \coloneqq \overline{\EE}_{\T}[\tau_A \mid Z_0 = z] \quad\text{ and } \quad\underline{h}^{\T}_A(z) \coloneqq \underline{\EE}_{\T}[\tau_A \mid Z_0 = z],
\end{align*}
respectively. 

There exists a characterization for upper expected hitting times similar to Equation~\eqref{eq:hittingtime_comp}.
Let $\overlineT:\R^{N}\to \R^{N}$ be the (possibly) nonlinear map defined for all $f\in \R^{N}$ and all $z\in \ZZ$ as 
\begin{equation*}
    [\overlineT f](z)\coloneqq\sup_{T\in \T}[Tf](z).
\end{equation*}
Then, we have that $\overline{h}^{\T}_A=\one_{A^c}+\one_{A^c}\cdot \overlineT \ \overline{h}^{\T}_A$, with a similar characterization for the lower expected hitting times~\cite{krak2019hitting}.

\subsubsection{Reachability Condition.} 
Given a set of transition matrices $\T$ we say that a set of states $\C\subseteq \ZZ$ is \textit{upper reachable}~\cite{debock2017imprecise, DeCooman2009} from $z\in \ZZ$ if there exists $n\in \N$ such that $ \ [\overlineT^n\one_{\C}](z)>0$, i.e. there exists $n\in \N$ and $T\in \T$ such that $[T^n \one_\C](z)>0$. 
We denote by $z\uppereach \C$ if $\C$ is upper reachable from $z$.
Analogously, the set $\C\subseteq \ZZ$ is \textit{lower reachable}~\cite{debock2025convergent, TJoensDeBock2021} from $z\in \ZZ$ if there exists $n\in \N$ such that $[\underlineT^n\one_\C](z)>0$.
We denote by $z\lowereach \C$ if $\C$ is lower reachable from $z$. 

Given a set of transition matrices $\T$, let $A\subset \ZZ$ be a closed target set, i.e. for all $x\in A, y\notin A$  and all $T\in \T$ it holds that $T(x,y)=0$.
We say that the reachability condition (R1) holds if
\[
\text{(R1)}:\quad \forall\,z\in A^c: z\lowereach A
\]
Under (R1), the upper expected hitting time \(\overline{h}^\T_A(z)\) is finite for all \(z\in A^c\)~\cite{krak2021comphit}.
If (R1) does not hold, then there exists a non-empty set of states $\A_\T\subset A^c$ such that
\begin{equation*}
   \A_\T:=\{z\in A^c \mid z\notlowereach A \}.
\end{equation*}
 For all states $z\in \A_T$ it naturally holds that $\overline{h}^\T_A(z)=+\infty$. This follows from the fact that, if $A$ is closed, then for all $z\in \A_\T$ there exists $T\in \T$ such that, for all $n\in \N$, $[T^n\one_A](z)=0$\footnote{This implication is not entirely trivial, but we omit the full derivation due to page limit constraints.}.

Let us denote with $\U_\T$ the set of all states of $A^c\setminus \A_\T$ that upper reach a state in $\A_\T$, i.e.
\begin{equation*}
    \U_\T \coloneqq\{z\in A^c\setminus \A_T \mid  z\uppereach \A_\T\}.
\end{equation*}
The following theorem characterizes upper hitting times for all states of $\ZZ$.
\begin{theorem}\label{th:rickcut}
    Let $\T$ be a set of transition matrices and let $\A_T$ and $\U_\T$ be as defined above. Then, for all $z\in \B_\T:=\A_\T \cup \U_\T$ it holds that $\overline{h}^\T_A(z)=+\infty$. Moreover, there exists a matrix $\tilde{T}\in \T$ such that 
    \[
    \overline{h}^\T_A|_{A^c\setminus \B_\T}=(I-\tilde{T}|_{A^c\setminus \B_\T})^{-1}{\textbf{1}},
    \]
    where $I$ denotes the identity matrix and \textbf{1} the vector of all ones. 
\end{theorem}
\begin{proof}
    It follows from Krak et al.~\cite[Theorem 12]{krak2019hitting} that there always exists a matrix $\tilde{T}\in \T$ such that $\overline{h}^\T_A=h^{\tilde{T}}$, where $h^{\tilde{T}}$ is the minimal non-negative solution of the linear system $h^{\tilde{T}}=\one_{A^c}+\one_{A^c}\cdot \tilde{T} h^{\tilde{T}}$. If $z\in \U_\T$, the Markov chain $\PP_{\tilde{T}}$ that starts in $z$ has positive probability of hitting the set $\A_\T$ thus positive probability of never reaching $A$.
    This implies that the upper expected hitting time for all $z\in \U_\T$ is infinite. 

    Observe that, for all $z\in A^c\setminus \B_\T$, for all $T\in \T$, and for all $n\in \N$, it holds that $[T^n \one_{\B_\T}](z)=0$.
    Therefore, we can restrict $\tilde{T}$ on $\tilde{\ZZ}:=\ZZ\setminus \B_\T$ as the Markov chain $\PP_{\tilde{T}}$ starting from any state in $\tilde{\ZZ}$ never leaves $\tilde{\ZZ}$.
    Thus, starting from any state in $\tilde{\ZZ}\setminus A$, the process eventually reaches $A$, and this yields the invertibility of $I-\tilde{T}|_{\tilde{\ZZ}\setminus A}$~\cite{norrisbook1997markov}.
    We are then able to conclude that $h^{\tilde{T}}|_{\tilde{\ZZ}\setminus A}=(I-\tilde{T}|_{\tilde{\ZZ}\setminus A})^{-1}{\textbf{1}}$ which is what we wanted.\qed
\end{proof}
An analogous result holds for lower expected hitting times: in this case, the definitions of $\A_\T$ and $\U_\T$ are obtained by swapping upper and lower reachability in their respective constructions.

\subsubsection{A Computational Method.}
Upper- and lower expected hitting times can be computed efficiently using iterative algorithms~\cite{krak2021comphit, krak2019hitting}. One such algorithm is the following: starting from any extreme point $T_1$ of (the convex set) $\T$, let $h_n$ be the unique solution of the linear system $h_n=\one_{A^c}+\one_{A^c}\cdot T_n h_n$ and let $T_{n+1}$ be an extreme point of $\T$ such that $\overline{T}h_{n}=T_{n+1}h_n$. Krak~\cite{krak2021comphit} showed that, under (R1), the sequence $(h_n)_{n\in\N}$ converges to $\overline{h}^{\T}_A$.
We refer to this procedure as Krak's algorithm. 

If there is a non-empty set $\A_\T$ then, as Theorem~\ref{th:rickcut} shows, we may restrict our analysis to $A^c\setminus \B_\T$ as elsewhere the upper expected hitting time is either zero or infinite. 
For each $T\in \T$ consider the restriction of $T$ on $\B_\T^c$ 
\begin{equation*}
    T|_{\B_\T^c} \;=\;\bigl(T{(i, j)}\bigr)_{\,i,j\in \B_\T^c}\;\in\R^{|\B_\T^c|\times|\B_\T^c|},
\end{equation*}
the submatrix of $T$ indexed by $\B_\T^c$. Then, the set of matrices $\T|_{\B_\T^c}$ containing the restriction on $\B_\T^c$ of all matrices in $\T$, satisfies (R1)\footnote{To compute the lower expected meeting time, we need to further restrict $\T$ excluding every matrix that has a positive probability of reaching $\B_\T$.}.
Therefore, we may apply Krak's algorithm to compute the upper expected hitting time for all nontrivial states of $\ZZ$.

Before applying Krak's algorithm, we need to first identify sets $\A_\T$ and $ \U_\T$ (in this order). To obtain the former we make use of the following lemma.
\begin{lemma}[\cite{debock2025convergent,debock2017imprecise,hermans2012characterisation}]
    Let $\C$ be a closed subset of $\ZZ$.
    Let $(\C_n)_{n\in \N_0}$ be the nondecreasing sequence given by $\C_0=\C$ and 
    \begin{equation*}
        \C_{n+1}=\C_n\cup \{z\in \ZZ\setminus \C_n \mid [\underlineT\one_{\C_n}](z)>0\},
    \end{equation*}
    for all $n\in \N_0$. Let $n^*$ be the first index for which $\C_{n^*}=\C_{n^*+1}$. Then, the set $\C_{n^*}$ contains all states of $\ZZ$ that lower reach $\C$. 
\end{lemma}
This lemma enables us to implement a straightforward algorithm to identify $\A_\T^c$ (and therefore $\A_\T$). This algorithm is a modification of Algorithm 2 in \cite{debock2017imprecise}.
\begin{algorithm}[ht]
\caption{Finding the set \(\A_{\T}\)}
\label{alg:lower-reachability}
\(\C \gets A\)\;
\Repeat{\(\mathcal{M} = \varnothing\)}{
  \(\mathcal{M} \gets \varnothing\)\;
  \ForEach{\(z \in \ZZ \setminus \C\)}{
    \If{\([\underline T\,\mathbf{1}_{\C}](z) > 0\)}{
      \(\mathcal{M} \gets \mathcal{M} \cup \{z\}\)\;
    }
  }
  \(\C \gets \C \cup \mathcal{M}\)\;
}
\Return{\(\A_\T=\C^c\)}\;
\end{algorithm}

Finding the set $\U_\T$ of states that upper reach $\A_\T$ is even more straightforward. First, we build the directed graph $G_{\overlineT}=(\ZZ, E)$, where $(x,y)\in E$ if $[\overlineT\one_{\{y\}}](x)>0$. Then, for all $z\in A^c\setminus \A_\T$, we check (e.g. by performing a breadth‑first search on \(G_{\overline T}\)) whether there exists a path in $G_{\overlineT}$ from $z$ to $\A_\T$. If such path exists then $z\in \U_\T$.

\section{Upper and Lower Meeting Times for Interdependent Stochastic Agents} \label{sec:upmeetingtime}

We now move on to the formalization of what we call \textit{stochastic agents}.
The basic idea is that we have agents---essentially random walkers on the state space $\ZZ$---whose behaviour can be influenced by making selections from some allowed set of possible actions.

Formally, for each $z\in\ZZ$, we consider a set $\mathcal{T}_z$ of probability mass functions on $\ZZ$, which we interpret as the set of allowed actions that an agent can select from, whenever they are in state $z$. 
We then say that a {stochastic agent} is a random walker on $\ZZ$ whose behaviour is determined by her specifying some selection that is compatible with these sets $\mathcal{T}_z$.
Such a selection, say $T(z,\cdot)\in\mathcal{T}_z$, thus corresponds to a probability distribution that governs the stochastic behaviour of the agent: the probability that she will be in state $z'$ at the next time point is given by $T(z,z')$.

To clean up the notation, we gather all these allowed actions in a set $\mathcal{T}$ of transition matrices, such that
\begin{equation}\label{eq:trans_mat_set_from_action_sets}
    \mathcal{T} = \bigl\{ T \text{ a trans. mat.}\,\vert\, T(z,\cdot)\in\mathcal{T}_z \text{ for all } z\in\ZZ\bigr\}\,,
\end{equation}
where $T(z,\cdot)$ denotes the $z$ row of the transition matrix $T$. In this way, any state-dependent selection can be expressed as a single element of $\mathcal{T}$.

Note that so far we have not specified how these selections are made and what they can depend on.
Indeed, a simple example would be an agent who simply selects some $T\in\mathcal{T}$ and maintains this forever; the associated walker can then be represented by a homogeneous Markov chain with transition matrix $T$. 
Another example might be an agent whose selections depend on time: these selections $(T_n)_{n\in \N_0}$ then determine a non-homogeneous Markov chain that describes the stochastic evolution of the agent.
However, the selections might also be functions of additional things.

In particular, we are interested in studying the multi-agent setting, in which the behaviour of agents is {interdependent}.
We start the exposition here with the simplest case, in which we are only dealing with two agents.
We say that they are \emph{interdependent}, when the selection that an agent makes at some point in time, depends on the state of the other agent at that same point in time (and vice versa).
For each state $y$ that one agent can be in, the other agent needs to specify a selection $T^y(x,\cdot)\in\mathcal{T}_x$ for each state $x$ that \emph{they} can be in.
Hence, the selection for this agent can be summarised as a collection $(T^y)_{y\in \ZZ}$ of transition matrices in $\mathcal{T}$.
Similarly, we denote the selection of the other agent as $(S^x)_{x\in \ZZ}$.

Given the selections of the two agents, their joint behaviour can be modelled by a stochastic process $(X_n,Y_n)$ on $\ZZ^2$, in which $X_n$ and $Y_n$ represent the states of the two agents at time $n$.
This process is Markovian and homogeneous, and satisfies
\begin{equation}\label{eq:definterdep}
    \begin{split}
        &\PP_{(X,Y)}\bigl((X_{n+1},Y_{n+1})=(x',y')\mid (X_n,Y_n)=(x,y)\bigr)=T^y(x,x')S^x(y,y')\,,
    \end{split}
\end{equation}
where we write $(T^y)$ and $(S^x)$ for the selections driving $X_n$ and $Y_n$, respectively.

To summarize the above construction, we have described a multi-agent setting, in which the behaviour of each agent is stochastic, but influenced by their selections from some set of allowed actions.
There are potentially many interesting inference problems that we could consider, but here we focus on the \emph{expected meeting time} for the two agents.

To conclude the formalization of the problem that we want to study, we now consider the case where we are \emph{epistemically uncertain} about the selections of these agents.
Clearly, this uncertainty affects our ability to compute these expected meeting times.
However, depending on the exact uncertainty model that we use to represent these beliefs, we may still be able to make useful conclusions about this quantity of interest.
In the remainder of this section, we consider this problem under various choices of such uncertainty models.

\subsection{Degenerate Belief}\label{subsec:degen_belief}

Let us begin by considering what is arguably the simplest uncertainty model: the degenerate belief, in which we are certain that the two agents use selections $(T^y)_{y\in \ZZ}$ and $(S^x)_{x\in \ZZ}$, respectively.\footnote{Clearly, in this case we are not actually ``uncertain'' at all, but this framing allows us to represent all cases that we consider in a unified framework.} 

It follows immediately that this case is described exactly by the single joint process $(X_n, Y_n)$ characterised in Equation~\eqref{eq:definterdep}.
Since this joint process is a homogenous Markov chain, it is determined by a transition matrix $\Delta$ on $\ZZ^2$.
We see from the above that the $((x,y),(x',y'))$ entry, which represents the probability of transitioning from $(x,y)$ to $(x',y')$, is given by
\begin{equation}\label{eq:degen_trans_mat_from_selections}
    \Delta\bigl((x,y),(x',y')\bigr)=T^y(x,x')S^x(y,y').
\end{equation}
Moreover, the expected meeting time for the agents coincides with the expected hitting time of the Markov chain $(X_n,Y_n)$ on target set $D=\{(z,z)\mid z\in \ZZ\}$.
The corresponding matrix of expected meeting times $m$ satisfies a system of equations analogous to \eqref{eq:meetingsys}, where each occurrence of \(T\) and \(S\) is replaced by \(T^y\) and \(S^x\), respectively; using the definition of $\Delta$ yields
\begin{equation}\label{eq:meetingDelta}
\begin{cases}
    m(x,x)=0 & \text{for all $x\in \ZZ$,}\\
    m(x,y)=1+\sum_{x',y'\in \ZZ}\Delta\bigl((x,y),(x',y')\bigr)m(x',y') & \text{for all $x\ne y$.}
\end{cases}
\end{equation}
So, in this case, we can compute the expected meeting times exactly, by solving this linear system of equations.

\subsection{Vacuous Belief}\label{subsec:vacuous}

The next uncertainty model that we consider is the \emph{vacuous belief} over the space of all possible selections; this corresponds to the case in which we are fully ignorant about which selections the agents use.

We know from our discussion in Section~\ref{subsec:degen_belief} that any specific pair of selections, say $(T^y)_{y\in \ZZ}$ and $(S^x)_{x\in \ZZ}$, determines a homogenous Markov chain on $\ZZ^2$ that is characterized by a transition matrix $\Delta$, as in Equation~\eqref{eq:degen_trans_mat_from_selections}. By following this construction for every possible pair of selections that the agents might choose, we induce a set of transition matrices on $\ZZ^2$:
    \begin{equation*}
        \T^2 \coloneqq \left\{\Delta\in \R^{{N}^2\times {N}^2}\,\Big\vert\, \Delta((x,y),(x',y')):=T^y(x,x')S^x(y,y')\right\}\,,
    \end{equation*}
    where $N=|\ZZ|$ and we are varying all possible selections $(T^y)$ and $(S^x)$ in $\T$. 
    
Each element $\Delta$ of $\T^2$ exactly corresponds to a pair of selections $(T^y)$ and $(S^x)$, which makes this set a convenient representation of the space over which we want to model our beliefs. In particular, each $\Delta\in\T^2$ induces a homogeneous Markov chain $\mathbb{P}_\Delta$, which we collect in the set $\smash{\mathcal{P}_{\T^2}}=\bigl\{\mathbb{P}_\Delta\,\vert\,\Delta\in\T^2\bigr\}$. Our vacuous beliefs are now represented by the upper (and lower) expectation operators given by
\begin{equation}\label{eq:def_vacuous_upper_lower_exp}
    \overline{\EE}_{\T^2}[\,\cdot\,]\coloneqq\sup_{\PP\in \Pcal_{\T^2}}\EE_{\PP}[\,\cdot\,],\qquad\text{and}\qquad \underline{\EE}_{\T^2}[\,\cdot\,]\coloneqq\inf_{\PP\in \Pcal_{\T^2}}\EE_{\PP}[\,\cdot\,]\,.
\end{equation}
We immediately recognize the set $\smash{\mathcal{P}_{\T^2}}$ as an \emph{imprecise Markov chain} on the product space $\ZZ^2$, that is characterised by the set $\T^2$ of transition matrices. Moreover, our vacuous beliefs in Equation~\eqref{eq:def_vacuous_upper_lower_exp} correspond simply to the upper- and lower expectation operators for this imprecise Markov chain.

This also means that we can leverage the known results from Section~\ref{subsec:IMCpreli}, provided that $\T^2$ satisfies the required regularity conditions; let us analyze how these carry over from the set $\T$ of allowed actions.
If $\T$ is compact and with SSR\footnote{Note that $\T$ always has SSR, when it is constructed as in Equation~\eqref{eq:trans_mat_set_from_action_sets}.}, the same holds for $\T^2$.
This follows from the fact that for any fixed $x$ and $y$, rows $T^y(x,\cdot)$ and $S^x(y,\cdot)$ appear only in one row of $\Delta$. Convexity does not carry over from $\T$ to $\T^2$.
This is due to the fact that the map
\[
(T^1,\dots,T^{N},\,S^1,\dots,S^{N})
\;\longmapsto\;
\Delta
\]
is bilinear 
(i.e. linear in each family \((T^y)\) or \((S^x)\) separately) 
rather than affine, and only affine maps are guaranteed to preserve convexity~\cite{rockafellar1970convex}. 
Restoring convexity by replacing \(\T^2\) with its convex hull will prove useful later. 

Moving on, we can now characterise the upper expected meeting time as
\begin{align}
    \overline{m}^{\T^2}(x,y):=\overline{\EE}_{\T^2}[\,\mu\mid (X_0,Y_0)=(x,y)]=\overline{h}^{\T^2}_D\bigl((x,y)\bigr). \label{eq:uppermeeting}
\end{align}
Since the supremum in Equation~\eqref{eq:uppermeeting} is always achieved on a extreme point~\cite{bauer1958} of $\T^2$, and since \(\T^2\) and its convex hull \(\ch(\T^2)\) share the same extreme points, the upper expected meeting time can be computed by  optimizing over \(\ch(\T^2)\), obtaining the same maximal value.

Similarly as before, let us define the upper transition operator $\overline{\Delta}:\R^{N^2}\to \R^{N^2}$ for $\ch(\T^2)$ as
\(
    [\overline{\Delta}f](z):=\sup_{\Delta\in \ch(\T^2)} [\Delta f](z),
\)
for all $f\in \R^{N^2}$ and all $z\in \ZZ^2$. Then, $\overline{m}^{\T^2}$ is the minimal non-negative solution to 
\begin{equation}\label{eq:lowermeetchara}
    \vect(\overline{m}^{\T^2})=\one_{D^c}+\one_{D^c}\cdot \overline{\Delta} \vect(\overline{m}^{\T^2}).
\end{equation}
An analogous construction characterises lower expected meeting times.

\subsubsection{Computing Upper Meeting Times} 
To compute upper (and lower) expected meeting times under our vacuous belief model, we may use Krak's algorithm~\cite{krak2021comphit} for upper (and lower) expected hitting times, restricted to $\B_{\T^2}^c$, as introduced at the end of Section~\ref{subsec:IMCpreli}. 
First of all, we need to identify the sets $\A_{\T^2}$ and $\U_{\T^2}$ that have infinite upper expected meeting time. 
To do so, we apply the algorithms from the end of Section~\ref{subsec:IMCpreli} to the product space \(\ZZ^2\), replacing \(\T\) with \(\ch(\T^2)\).

\subsubsection{The Corresponding Optimal Control Problem}

We feel that it is relevant to note that while we phrased the problem here in terms of our epistemic uncertainty about the selections of the agents, our formalization has a natural other interpretation as an optimal control problem. That is, the agents might be strategically choosing their selections---or even have these selections imposed by some external controller---in such a way as to minimize or maximize their meeting time, subject to the constraints on the allowed actions $\T$. Finding the optimal pair of selections then amounts to finding the elements of $\T^2$ that attain the extremal values in Equation~\eqref{eq:def_vacuous_upper_lower_exp}. Indeed, Krak's algorithm~\cite{krak2021comphit} may be used precisely to find such an optimal control policy.

\subsubsection{Exploiting the Symmetry.} We shall note that we can exploit the symmetry of the problem and restrict our analysis on the \textit{symmetric product space} $\ZZ^2/S_2$, where $S_2$ is the symmetric group~\cite{Rotman1995}.
Explicitly, we identify each pair $(x,y)$ with $(y,x)$, as the two agents are indistinguishable, so that the space $\ZZ^2/S_2$ consists of all unordered pairs of states of $\ZZ$.
As a consequence the optimal selections $(T^y)_{y\in \ZZ}$ and $(S^x)_{x\in \ZZ}$ coincide, i.e. $T^z=S^z$ for all $z\in \ZZ$.  
As we pass to the symmetric product space \(\ZZ^2/S_2\), \(\T^2\) naturally becomes the set of transition matrices over \(\ZZ^2/S_2\). 
For brevity, we continue to denote it by \(\T^2\); the intended domain---whether the full product space or its symmetric quotient---is always clear from context.

\begin{example}\label{ex:doublegraph}
The figure below shows the graph $G_{\overlineT}$ that encodes all possible transitions on the state space \(\ZZ = \{1,2,3,4,5\}\)  (left) and the corresponding transitions on the symmetric product space $\ZZ^2/S_2$ (right). An arrow that connects node $i$ to node $j$ in the left graph indicates that $[\overlineT\one_{\{j\}}](i)>0$; an analogous interpretation applies for the right graph with $\overline{\Delta}$ in place of $\overlineT$.
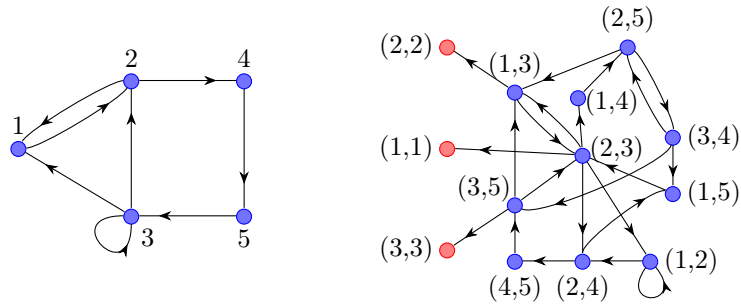
    \begin{center}
    \begin{tikzpicture}[
        xscale=1.5,yscale=1.5,
        every node/.style={draw=black,circle,fill=black!70,inner sep=2pt},
        redNode/.style={draw=red, fill=red!50},
        blueNode/.style={draw=blue, fill=blue!55},
        blackNode/.style={draw=black, fill=black},
        every label/.style={rectangle,fill=none,draw=none},
        every edge/.style={draw,bend right,looseness=0.3,
            postaction={
                decorate,
                decoration={
                    markings,
                    mark=at position 0.89 with {\arrow[#1]{Stealth}}
                }
            }
        }
    ]
        \node[blueNode,label=90:1] (1) at (4,55.3) {};
        \node[blueNode,label=90:2] (2) at (5,55.9) {};
       \node[blueNode,label=300:3] (3) at (5,54.7) {};
        \node[blueNode,label=90:4] (4) at (6,55.9) {};
        \node[blueNode,label=270:5] (5) at (6,54.7) {};
        \path (1) edge (2);
        \path (2) edge (1);
        \path (3) edge[bend left=0] (1);
        \path (3) edge[bend left=0] (2);
        \path (3) edge[relative=false,out=180,in=270,min distance=4ex] (3);
        \path (5) edge[bend left=0] (3);
        \path (4) edge[bend left=0] (5);
        \path (2) edge[bend left=0] (4);
        
        \node[redNode,label=180:{(3,3)}] (33) at (7.8,54.4) {};
        \node[redNode,label=180:{(2,2)}] (22) at (7.8,56.2) {};
        \node[redNode,label=180:{(1,1)}] (11) at (7.8,55.3) {};
        \node[blueNode,label={[xshift=0.64ex, yshift=-0.8ex]130:{(3,5)}}] (35) at (8.4,54.8) {};
        \node[blueNode,label=270:{(4,5)}] (45) at (8.4,54.3) {};
        \node[blueNode,label=270:{(2,4)}] (24) at (9.6,54.3) {};
        \node[blueNode,label={[xshift=0.649ex, yshift=-0.35ex]180:{(1,2)}}] (12) at (9,54.67) {};
        \node[blueNode,label=0:{(1,5)}] (15) at (9.8,54.9) {};
        \node[blueNode,label=0:{(3,4)}] (34) at (9.8,55.4) {};
        \node[blueNode,label=90:{(2,5)}] (25) at (9.4,56.2) {};
        \node[blueNode,label=90:{(1,3)}] (13) at (8.4,55.8) {};
        \node[blueNode,label={[xshift=-0.649ex, yshift=-0.2ex]0:{(1,4)}}] (14) at (8.96,55.75) {};
        \node[blueNode,label={[xshift=-1.05ex, yshift=-0.72ex]45:{(2,3)}}] (23) at (9,55.24) {};
        \path (13) edge[bend right=10] (12);
        \path (23) edge[bend right=33] (34);
        \path (13) edge[bend left=0] (22);
        \path (23) edge[bend left=0] (11);
        \path (35) edge[bend left=0] (33);
        \path (35) edge[bend left=0] (13);
        \path (23) edge (13);
        \path (13) edge (23);
        \path (35) edge[bend left=0] (23);
        \path (45) edge[bend left=0] (35);
        \path (24) edge[bend right=0] (45);
        \path (23) edge[bend left=0] (24);
        \path (12) edge[bend left=0] (24);
         \path (23) edge[bend left=0] (12);
         \path (12) edge[relative=false,out=240,in=305,min distance=2.6ex] (12);
        \path (24) edge[bend left=0] (15);
        \path (15) edge[bend left=0] (23);
        \path (34) edge[bend left=0] (15);
        \path (34) edge[bend left=40] (35);
        \path (25) edge[bend left] (34);
        \path (34) edge[bend left] (25);
        \path (25) edge[bend left=0] (13);
        \path (14) edge[bend left=0] (25);
        \path (23) edge[bend left=0] (14);
    \end{tikzpicture}
    \end{center}
Depending on the choice of the set of allowed actions $\T$, some pair of states in $\ZZ$ may or may not have infinite upper expected meeting time. For example, if the set $\T$ is such that $[\underlineT\one_{\{4\}}](2)=0$, i.e an agent in state $2$ can choose not to go to state $4$, then the pair $(1,2)$ would become an absorbing state for $\ZZ^2/S_2$, and naturally $(1,2)\in \A_{\T^2}$. In this case, every pair of states, e.g. $(2,3)$, that has positive lower probability of eventually hitting $(1,2)$ would be part of the set $\U_{\T^2}$ whose states also have infinite upper expected meeting time.

\end{example}

\subsection{Degenerate-Vacuous Mixtures}

The final uncertainty model that we consider is the \emph{degenerate-vacuous mixture}. Note that this is a special case of the well-known \emph{linear-vacuous mixture}~\cite{DesterckeDubois2014}, in which the linear component is represented by a degenerate (point mass) belief. Consequently, to characterize this model we need to specify both a pair of selections---which we conveniently represent through a single transition matrix $\Delta\in\T^2$---and some numerical parameter $\epsilon\in[0,1]$. The associated upper expectation operator is then
\begin{equation}
    \overline{\mathbb{E}}_{\epsilon,\Delta, \T^2}[\cdot] \coloneqq (1-\epsilon)\mathbb{E}_{\mathbb{P}_\Delta}[\cdot] + \epsilon \overline{\mathbb{E}}_{\T^2}[\cdot]\,,
\end{equation}
and similarly for the lower expectation. Since for the particular case of meeting times, both terms on the right-hand side of this expression can be computed using the results in Sections~\ref{subsec:degen_belief} and~\ref{subsec:vacuous} respectively, it follows that we can also evaluate the upper (and lower) expected meeting times under this model.

We note that this uncertainty model may be particularly useful for conservatively evaluating performance in an ``unreliable control'' setting. That is, we might consider the degenerate component $\Delta$ to represent some optimal control, e.g. a minimizing pair of selections for the meeting time of the two agents, but in which these agents sometimes deviate from the intended control in a manner that is consistent with the allowed actions $\T$, yet otherwise completely unknown to us. The parameter $\epsilon$ measures this degree of unreliability, and the associated upper expected hitting time then provides a conservative estimate for the attained performance of this system.

\section{Meeting Times for Multiple Interdependent Agents}\label{sec:manyagents}
In this section, we extend the meeting‐time problem to multiple interdependent stochastic agents seeking to all gather at a single common location.

\subsection{Imprecise Markov Chain Formulation for $k$ Stochastic Agents}
Consider $k\in \N, k>2$ interdependent stochastic agents on the same state space $\ZZ$, with $|\ZZ|=N$, and with set of transition matrices $\T$. Each agent looks at the position of all other agents to determine which selection to make. For every agent $j\in [k]:=\{1,\dots, k\}$, her selections are identified with a collection of ${N}^{k-1}$ transition matrices, $(T_j^{(z_i)_{i\in [k], i\ne j}})$ where $(z_i)_{i\in [k], i\ne j}$ represents the states of all other agents.
To find the upper (and lower) expected meeting time of these stochastic agents, we consider the joint homogenous Markov chain $(Z^1,\dots, Z^k)$ on
$\ZZ^k=\ZZ\times\dots\times \ZZ$, $k$ times.
Similarly as before, we define a set of transition matrices $\T^k$ on the space $\ZZ^k$ as follows
\begin{equation*}
    \T^k\coloneqq\Bigl\{\Delta_k\in \R^{{N}^k \times {N}^k} \,\Big\vert\,\Delta_k((z_{1:k}),(z'_{1:k}))=\prod_{j=1}^k T_j^{(z_i)_{i\ne j}}(z_j,z'_j)\Bigr\},
\end{equation*}
where, for each $j=1,\dots, k$, we vary over all possible selections of the $k$ agents. 
Thus, each matrix $\Delta_k\in \T^k$ is defined using $k{N}^{k-1}$ transition matrices on $\ZZ$. 

As before, under the assumption that $\T$ is compact, convex and has SSR, the set $\T^k$ is compact and has separately specified rows. The upper (and lower) expected meeting time of $k$ interdependent stochastic agents with set of transition matrices $\T$ can be seen as the upper (and lower) hitting time of an imprecise Markov chain on the space $\ZZ^k$ with set of transition matrices $\T^k$ and target set $D^k:=\{(z,\dots,z) \mid z\in \ZZ\}$.
As in the case with two agents, by taking the convex hull of \(\T^k\) and defining the corresponding upper (and lower) transition operator, we obtain a characterization of the upper (and lower) meeting time analogous to that in Equation~\eqref{eq:lowermeetchara}.
Moreover, this also allows us to apply Krak's algorithm~\cite{krak2021comphit} to compute the upper (and lower) expected meeting time, as taking the convex hull leaves the upper (and lower) expected meeting time unchanged. 

\subsection{Symmetry Reduction}
The state space $\ZZ^k$ grows exponentially with $k$.
Exploiting the symmetry drastically reduces  the number of states to consider as 
\begin{equation*}
|\ZZ^k/S_k|= \binom{N+k-1}{k}\ll {N}^k=|\ZZ^k|,
\end{equation*}
where $S_k$ is the symmetric group of degree $k$. In other words, in the space $\ZZ^k/S_k$ all permutations of a $k$-tuple are identified.
As in the case $k=2$, every matrix $\Delta_k:\R^{|\ZZ^k/S_k|}\to \R^{|\ZZ^k/S_k|}$ is now defined using only $N^{k-1}$ matrices from $\T$, as the $k$ families of transition matrices coincide. 

Moreover, if $2\le i<k$ stochastic agents are in the same state $z\in \ZZ$ their selected action that minimizes the meeting time is the same for all $i$ agents and depends only on the position of the other $k-i$ agents.
Similarly, these other $k-i$ agents have the same optimal selection whether in $z$ there are 1 or $i$ distinct agents. 

Hence, if we want to compute the lower expected meeting time of $k>N$ interdependent stochastic agents, we can reduce to the same problem with $N$ such agents. From the optimal matrix $\Delta_N$ we can obtain the optimal selections for each agent, which we can then use to build the optimal matrix $\Delta_k$.
Therefore, the cost to compute lower expected meeting times for $k>N$ interdependent stochastic agents is equivalent to that of $N$ such agents, thus drastically reducing the computational burden for large values of \( k \).  

\section{Conclusion and Future Work}\label{sec:conclusion}
In this work we considered the meeting-time problem for multiple interdependent stochastic agents, about whose selected actions we are epistemically uncertain. 
By establishing that these meeting times correspond to hitting the diagonal in the product space, and by focusing on specific uncertainty models for our beliefs about these agents' selected actions, we were able to recast the problem using the existing theory of imprecise Markov chains.
We used their associated upper expectation operators to derive explicit characterizations for these quantities of interest, and leveraged Krak’s iterative algorithm~\cite{krak2021comphit} to compute them.
Finally, we showed that our approach extends naturally to an arbitrary number of agents via a $k$-fold product construction, and we exploited the problem’s inherent symmetries to mitigate the exponential growth of the state space.

In future work, we hope to investigate more general interdependency mechanisms between these agents, provide a more in-depth characterisation of the joint imprecise process that describes their behaviour, and consider the case where agents may have different---possibly adversarial---objectives. It would also be interesting to generalize the problem to a continuous-time setting.

\begin{credits}
\subsubsection{\ackname} 
This work has been partly supported by the PersOn project
(P21-03), which has received funding from the Nederlandse Organisatie
voor Wetenschappelijk Onderzoek (NWO).
Moreover, we are grateful to the three anonymous reviewers for their careful reading and thoughtful comments.
Their input was essential in helping us strengthen the paper and present our ideas more clearly.
\end{credits} 
\section*{Errata}
In the published version of this paper \cite{sangalli2025meeting} the requirement that the target set $A$ must be closed was omitted from Section 2.3, “Imprecise Markov Chains”.
In this revised version of the paper we have added a formal definition of a closed set, and updated the surrounding text accordingly. Two missing edges have now been added to the right-hand graph in Example \ref{ex:doublegraph}. Moreover, some statements in the “Symmetry Reduction” section are true only for lower meeting times, thus also that part has been corrected.
\sloppy
\printbibliography

\end{document}